\documentclass[11pt]{ip-journal}
\usepackage{amssymb, amsfonts, amsthm, mathrsfs}
\usepackage[super]{nth}
\usepackage[pagebackref,colorlinks=true]{hyperref}
\usepackage{graphicx, xcolor}
\usepackage[all]{xy}
\usepackage[leqno]{amsmath}

\theoremstyle{definition}
\newtheorem{thm}{Theorem}[section]

\newtheorem{cor}[thm]{Corollary}
\newtheorem{prop}[thm]{Proposition}

\newtheorem{rmk}[thm]{Remark}

\renewcommand{\labelenumi}{(\alph{enumi}).}

\makeatletter

\def\id{\mathrm{id}}

\def\tr{\mathrm{Tr}}

\def\ccc{\mathbb{C}}
\def\C{\mathbb{C}}

\def\rr{\mathbb{R}}
\def\R{\mathbb{R}}
\def\pp{\mathbb{P}}

\def\clo{\mathcal{O}}

\def\frI{\mathfrak{I}}

\def\pt{\partial}
\def\p{\partial}
\def\bpt{\bar{\pt}}
\def\ddb{\pt\bpt}
\def\dbar{\bar{\pt}}

\def\ud{\mathrm{d}}

\def\bz{\bar{z}}

\def\bzeta{\bar{\zeta}}
\def\bmu{\bar{\mu}}

\def\vol{\mathrm{vol}}

\def\spann{\mathrm{Span}}

\def\ind{\mathrm{Ind}}
\makeatother

\begin{document}

\title{\textbf{A Construction of Infinitely Many Solutions to the Strominger System}}
\author{Teng Fei, Zhijie Huang and Sebastien Picard}

\date{}

\maketitle{}

\tableofcontents

\section{Introduction}

The Strominger system \cite{hull1986b, strominger1986} is a system of partial differential equations characterizing the compactification of heterotic superstrings with torsion. Mathematically speaking, we may think of the Strominger system as a generalization of Ricci-flat metrics on non-K\"ahler Calabi-Yau 3-folds, which is simultaneously coupled with the Hermitian Yang-Mills equation on a gauge bundle.

Let $X$ be a complex 3-fold, preferably compact, with holomorphically trivial canonical bundle. We fix a nowhere vanishing holomorphic (3,0)-form $\Omega$ on $X$. Let $\omega$ be a Hermitian metric on $X$, and denote by $\|\Omega\|_{\omega}$ the norm of $\Omega$ with respect to the metric $\omega$. In addition, let $(E,h)$ be a holomorphic vector bundle over $X$ equipped with a Hermitian metric. We denote by $R$ and $F$ the endomorphism-valued curvature 2-forms of the holomorphic tangent bundle $T^{1,0}X$ and $E$ respectively. Finally, let $\alpha'\in\rr$ be a constant. We can write down the Strominger system as follows \cite{li2005}:
\begin{eqnarray}
\label{hym}F\wedge\omega^2=0,\quad F^{0,2}=F^{2,0}=0,\\
\label{ac}i\pt\bpt\omega=\frac{\alpha'}{4}(\tr(R\wedge R)-\tr(F\wedge F)),\\
\label{cb}\ud(\|\Omega\|_\omega\cdot\omega^2)=0.
\end{eqnarray}
In the literature, Equations (\ref{hym}), (\ref{ac}) and (\ref{cb}) are known as the Hermitian Yang-Mills equation, the anomaly cancellation equation and the conformally balanced equation respectively. In this paper, we will only use the Chern connection to compute the curvatures $R$ and $F$ and work with positive $\alpha'$.

It is not hard to see that when $\omega$ is a K\"ahler metric, by Yau's theorem \cite{yau1978} the full Strominger system can be solved if we take $\omega$ to be Ricci-flat and simultaneously embed the spin connection into the Yang-Mills connection, meaning to set $R=F$. Such a solution corresponds to the torsion-free compactification of superstrings \cite{candelas1985}. However, the Strominger system allows for more general backgrounds than K\"ahler Calabi-Yau manifolds. In fact, the interest of the Strominger system is to solve it on non-K\"ahler Calabi-Yau 3-folds, where we say that a complex manifold is \emph{non-K\"ahler} if it does not admit any K\"ahler metric at all. Currently the only constraint is that $X$ needs to admit a balanced metric \cite{li2005}, i.e., a Hermitian metric whose K\"ahler form is co-closed.

Unfortunately, there are not many non-K\"ahler solutions to the Strominger system known so far. It appears to the authors that the only compact examples are the notable Fu-Yau solution \cite{fu2007,fu2008, phong2017c} and a few parallelizable examples, see \cite{fernandez2009, grantcharov2011, fei2015, otal2017} and references therein. In terms of solutions on noncompact spaces, we also have local models constructed in \cite{fu2009, fernandez2014, fei2015b, halmagyi2016} etc.

As for compact K\"ahler Calabi-Yau manifolds (treated as solutions to the Strominger system), it is widely speculated that in each dimension there are only finitely many deformation types and hence finitely many sets of Hodge numbers. Moreover, there are no explicit expressions for Calabi-Yau metrics except for the flat case.

In this paper, we demonstrate that the non-K\"ahler world of solutions to the Strominger system is considerably different. More precisely, we construct explicit smooth solutions to the Strominger system on compact non-K\"ahler Calabi-Yau 3-folds with infinitely many topological types and sets of Hodge numbers. The following is a summary of our construction:

\begin{thm}~\\
Let $\Sigma$ be a compact Riemann surface of genus $g \geq 3$ with a basepoint-free theta characteristic. Let $M$ be a compact hyperk\"ahler 4-manifold. The generalized Calabi-Gray construction gives rise to a compact non-K\"ahler Calabi-Yau 3-fold $X$, which is the total space of a fibration $p: X \rightarrow \Sigma$ with fiber $M$, admitting explicit smooth solutions to the Strominger system with gauge bundle $E= \Omega_{X / \Sigma}$ taken to be the relative cotangent bundle of the fibration. If $M=T^4$, we may also take $E$ to be any flat vector bundle.
\end{thm}

Recently D.H. Phong, X.-W. Zhang and the third-named author showed in a series of papers \cite{phong2018b, phong2018d, phong2018e, phong2017b} that the Anomaly flow is a powerful tool in studying the Strominger system. Using our ansatz, the Anomaly flow reduces to an interesting parabolic equation on Riemann surfaces. This aspect is investigated in \cite{fei2017d}.
\medskip
\par \noindent\textbf{Acknowledgements} The authors would like to thank Prof. D.H. Phong and Prof. S.-T. Yau for their constant encouragement and help. The authors are also indebted to R.M. Schoen, L.-S. Tseng, B.-S. Wu and X.-W. Zhang for useful discussions. The third-named author was supported in part by National Science Foundation grant DMS-12-66033.

\section{Reduction of the Strominger System and Its Solutions}

\subsection{The Generalized Calabi-Gray Construction}\label{construction}

To obtain interesting solutions to the Strominger system, we need first to construct compact non-K\"ahler Calabi-Yau 3-folds with balanced metrics. A class of nice such examples is what we call the generalized Calabi-Gray construction \cite{fei2016}, which can be phrased as follows.

Let $(M,g)$ be a compact hyperk\"ahler 4-manifold. It is well-known that $M$ is either a flat 4-torus or a K3 surface with a Calabi-Yau metric. We denote by $I$, $J$ and $K$ the compatible complex structures with $IJK=-\id$ and write $\omega_I$, $\omega_J$ and $\omega_K$ for the corresponding K\"ahler forms, defined as usual by $\omega_I(v_1,v_2)=g(Iv_1,v_2)$, $\omega_J(v_1,v_2)=g(Jv_1,v_2)$ and $\omega_K(v_1,v_2) = g(Kv_1,v_2)$. In fact, for any real numbers $\alpha$, $\beta$ and $\gamma$ such that
\[\alpha^2+\beta^2+\gamma^2=1,\]
we have a compatible complex structure $\alpha I+\beta J+\gamma K$, whose associated K\"ahler form is given by $\alpha\omega_I+\beta\omega_J+\gamma\omega_K$.

Following \cite{hitchin1987}, let us first describe the twistor space $Z$ of $M$. By stereographic projection, we may parameterize $S^2=\{(\alpha,\beta,\gamma)\in\rr^3:\alpha^2+\beta^2+\gamma^2=1\}$ by $\zeta\in\ccc\pp^1$ via
\begin{equation}\label{stereographic}
(\alpha,\beta,\gamma)= \left(\frac{1-|\zeta|^2}{1+|\zeta|^2},\frac{\zeta+\bzeta}{1+|\zeta|^2},\frac{i(\bzeta-\zeta)}{1+|\zeta|^2}\right).
\end{equation}
Moreover, we always equip $\ccc\pp^1\cong S^2$ with the round (Fubini-Study) metric
\begin{equation}\label{fs}
\omega_{FS}=\frac{2i\ud\zeta\wedge\ud\bzeta}{(1+|\zeta|^2)^2}.
\end{equation}
The twistor space $Z$ of $M$ is defined to be the manifold $Z=\ccc\pp^1\times M$ with the tautological almost complex structure $\frI$ given by
\begin{equation}\label{hypertwistor}
\frI=j\oplus(\alpha I_x+\beta J_x+\gamma K_x)
\end{equation}
at point $(\zeta,x)\in\ccc\pp^1\times M$, where $j$ is the standard complex structure on $\ccc\pp^1$ with holomorphic coordinate $\zeta$.

In fact, we have
\begin{thm}\cite{atiyah1978, hitchin1981, michelsohn1982, hitchin1987, muskarov1989}\label{twistor}
\begin{enumerate}
\item $\frI$ is an integrable complex structure, making $(Z,\frI)$ a compact non-K\"ahler 3-fold whose natural product metric is balanced.
\item The natural projection $\pi:Z=\ccc\pp^1\times M\to\ccc\pp^1$ is holomorphic.
\item Let $\Lambda^2\Omega_{Z/\ccc\pp^1}$ be the determinant line bundle of the relative cotangent bundle associated to the holomorphic fibration $\pi$. Then the line bundle $\Lambda^2\Omega_{Z/\ccc\pp^1}\otimes\pi^*\clo(2)$ on $Z$ has a global section which defines a holomorphic symplectic form on each fiber of $\pi$.
\end{enumerate}
\end{thm}

Now let $\Sigma$ be a compact Riemann surface of genus $g$ and let $\varphi:\Sigma\to\ccc\pp^1$ be a nonconstant holomorphic map. We may treat $\varphi$ either as a nonconstant meromorphic function $\zeta$ on $\Sigma$ or, by identifying $\ccc\pp^1$ with the unit sphere in $\rr^3$, as a map $\varphi=(\alpha,\beta,\gamma)$ into $\rr^3$. These two viewpoints are related by the stereographic projection formula (\ref{stereographic}). By pulling back the holomorphic fibration $\pi:Z\to\ccc\pp^1$ over $\varphi:\Sigma\to\ccc\pp^1$, we get a holomorphic fibration $p:X=\varphi^*Z\to\Sigma$. As a complex manifold, $X$ is topologically $\Sigma\times M$ with a twisted complex structure $J_0 = j_\Sigma \oplus(\alpha I_x+\beta J_x+\gamma K_x)$. We have
\begin{prop}\cite{fei2015b, fei2016, fei2016b}
\begin{enumerate}
\item $X$ has trivial canonical bundle if and only if
\begin{equation}\label{spinor}
\varphi^*\clo(2)\cong K_\Sigma,
\end{equation}
where $K_\Sigma$ is the canonical bundle of $\Sigma$.
\item Under (a), $X$ is non-K\"ahler with balanced metrics.
\end{enumerate}
\end{prop}

From now on, we shall call condition (\ref{spinor}) plus that $\varphi$ is not the constant map the ``pullback condition''. Assuming the pullback condition, it is clear that $S=\varphi^*\clo(1)$ is a square root of $K_\Sigma$, which is known as a theta characteristic in algebraic geometry, or a spin structure according to Atiyah \cite{atiyah1971}. Furthermore, the linear system associated to the line bundle $S$ is basepoint-free.

Conversely, if we start with a basepoint-free theta characteristic $S$ on $\Sigma$, we may choose $s_1,s_2\in H^0(\Sigma,S)$ such that $s_1$ and $s_2$ have no common zeroes, then $\zeta=s_1/s_2$ is a meromorphic function on $\Sigma$. Moreover, $\zeta$ defines a holomorphic map $\varphi:\Sigma\to\ccc\pp^1$ such that the pullback condition holds.

Therefore in this paper, we shall call a pair $(\Sigma,\varphi)$ such that the pullback condition holds a \emph{vanishing spinorial pair}. The generalized Calabi-Gray construction says that from any vanishing spinorial pair, one can construct a compact non-K\"ahler Calabi-Yau 3-fold $X$ with balanced metrics. In particular, it includes the classical construction of Calabi \cite{calabi1958} and Gray \cite{gray1969}, where $\Sigma$ is an immersed minimal surface in a flat 3-torus and $\varphi$ its Gauss map. We also say such an $X$ is a \emph{genus $g$ generalized Calabi-Gray manifold}, where $g$ is the genus of $\Sigma$.

It is trivial to remark that if $(\Sigma,\varphi)$ is a vanishing spinorial pair and $\tau:\widetilde{\Sigma}\to\Sigma$ an unramified covering map, then $(\widetilde{\Sigma},\varphi\circ\tau)$ is also a vanishing spinorial pair.

\subsection{Minimal Surfaces in Flat 3-Spaces}

In this section, we construct vanishing spinorial pairs by studying minimal surfaces in flat $T^3$. Let $T^3=\R^3/\Gamma$ be a 3-torus equipped with the standard flat metric. Let $x=(x_1,x_2,x_3)$ be the coordinate function on $\R^3$, which we also use as local flat coordinates on $T^3$. Now assume that $\Sigma\subset T^3$ is a closed (immersed) oriented surface of genus $g$ with induced metric and let $(u,v)$ be local isothermal coordinates on $\Sigma$. It is well-known that the (conformal class of) induced metric on $\Sigma$ determines a complex structure on $\Sigma$ and $z=u+iv$ is a local holomorphic coordinate. The metric on $\Sigma$ can be expressed as
\[\ud s^2=\rho(u,v)(\ud u^2+\ud v^2)=\rho(z,\bz)\ud z\ud\bz.\]
In other words,
\[\langle{\p x\over \p u},{\p x\over \p u}\rangle=\langle{\p x\over\p v},{\p x\over\p v}\rangle=\rho\textrm{\quad and \quad}\langle{\p x\over\p u},{\p x\over\p v}\rangle=0,\]
where $\langle\cdot,\cdot\rangle$ is the ambient metric on $T^3$. Moreover we can write down the K\"ahler form $\hat{\omega}$ as
\[\hat{\omega}=\rho(u,v)\ud u\wedge\ud v=\frac{i\rho(z,\bz)}{2}\ud z\wedge\ud\bz.\]
Locally, define
\[\phi=(\phi_1,\phi_2,\phi_3)={\p\over\p z}(x_1,x_2,x_3)=\frac{1}{2}({\p\over\p u}-i{\p\over\p v})(x_1,x_2,x_3).\]
The isothermal condition implies that
\[\phi_1^2+\phi_2^2+\phi_3^2=0.\]
It is a well-known result that $\Sigma$ is a minimal surface if and only if $\phi$ is holomorphic, or equivalently, $x=(x_1,x_2,x_3)$ is harmonic with respect to the Laplace-Beltrami operator
\[\Delta=\frac{1}{\rho}({\p^2\over\p u^2}+{\p^2\over\p v^2})=\frac{4}{\rho}{\p^2\over\p z\p\bz}\]
on $\Sigma$. Notice that for any smooth function $u$ on $\Sigma$, we have
\[2i\ddb u=\Delta u\cdot\hat{\omega}.\]

Now let us assume that $\Sigma$ is minimal and denote by $Q$ the Fermat quadric
\[Q=\{[\phi_1:\phi_2:\phi_3]\in\C\pp^2:\phi_1^2+\phi_2^2+\phi_3^2=0\}.\]
Therefore we get a map $\nu:\Sigma\to Q$ given by
\[z=u+iv\mapsto[\phi_1(z):\phi_2(z):\phi_3(z)],\]
which is a globally defined holomorphic map and it does not depend on the choice of local isothermal coordinates. This is known as the tangential Gauss map.

A simple genus calculation indicates that $Q$ is biholomorphic to $\C\pp^1$. Let $\clo(1)$ be the positive generator of the Picard group of $Q$ and let $H$ be the hyperplane line bundle on $\C\pp^2$. It is easy to see that
\[H|_Q\cong\clo(2).\]
Moreover each $\phi_j$ can be thought of as a section of $H$, which corresponds to a globally defined holomorphic 1-form $\mu_j=\phi_j\ud z$ on $\Sigma$. From this, we see that
\[\nu^*H\cong K_\Sigma,\]
in other words, $(\Sigma,\nu)$ is a vanishing spinorial pair.

In fact, an explicit identification $g:\C\pp^1\to Q$ can be constructed as follows
\[[z_1:z_2]\mapsto [\phi_1:\phi_2:\phi_3]=[2z_1z_2:z_2^2-z_1^2:-i(z_1^2+z_2^2)].\]
Write $\zeta=z_2/z_1$, then
\[\zeta=\frac{\phi_2+i\phi_3}{\phi_1}=-\frac{\phi_1}{\phi_2-i\phi_3}.\]
Through the stereographic projection (\ref{stereographic}), it follows that the composition $\varphi=g^{-1}\circ\nu:\Sigma\to \C\pp^1=S^2$ is exactly the classical Gauss map given by the unit normal vector field. Moreover, the Gauss-Bonnet Theorem tells us that $\varphi:\Sigma\to \C\pp^1$ is of degree $g-1$.

According to the work of Meeks \cite{meeks1990} and Traizet \cite{traizet2008}, for every $g\geq3$, there exist minimal surfaces of genus $g$ in $T^3$. Conversely, it is an easy exercise in algebraic geometry that the existence of vanishing spinorial pair implies that the genus is at least three.

A pair $(\Sigma,\varphi)$ with pullback condition is not so far from being minimal. It turns out that (\ref{spinor}) provides the Weierstrass data for the universal cover of $\Sigma$ to be minimally immersed into $\rr^3$ such that its Gauss map is given by $\varphi$ itself. In this sense, we may think of $(\Sigma,\varphi)$ with pullback condition satisfied as ``minimal surfaces in $T^3$'' without solving the period problem.

\subsection{Scalar Reduction of the Anomaly Equation}

Starting from a vanishing spinorial pair $(\Sigma,\varphi)$, we can construct a compact non-K\"ahler Calabi-Yau 3-fold $X$ as described in Subsection \ref{construction}. Now let us consider solving the Strominger system on $X$, following the procedure and calculations given in \cite{fei2017b, fei2016}.

We start with the conformally balanced equation (\ref{cb}). Identifying $\ccc\pp^1$ with $Q$, we obtain a holomorphic map
\[g\circ\varphi:\Sigma\to Q\]
such that
\[(g\circ\varphi)^*(H|_Q)\cong K_\Sigma.\]
By pulling back $\phi_1$, $\phi_2$ and $\phi_3$ as sections of $H|_Q$, we get holomorphic 1-forms $\mu_1$, $\mu_2$ and $\mu_3$ on $\Sigma$ such that
\[\mu_1^2+\mu_2^2+\mu_3^2=0\]
and that the linear system spanned by $\{\mu_j\}_{j=1}^3$ is basepoint-free. As is known in the literature relating to Theorem \ref{twistor}(c), one can check that
\[\Omega=\mu_1\wedge\omega_I+\mu_2\wedge\omega_J+\mu_3\wedge\omega_K\]
is a nowhere vanishing holomorphic (3,0)-form on $X$. Indeed, using a suitable local coordinate $z$ on $\Sigma$, we have the local expression
\[ \Omega = 2 \varphi \, \ud z \wedge \omega_I + (\varphi^2 -1) \, \ud z \wedge \omega_J - i(1+\varphi^2) \, \ud z \wedge \omega_K. \]
A similar expression arises using $\xi=z_1/z_2$ on $\C\pp^1$. A computation shows that $\Omega(\p_z, v, x + i J_0 x)=0$ for all $v,x \in TM$.

Next, the expression
\[\widehat{\omega}=i(\mu_1\wedge\bmu_1+\mu_2\wedge\bmu_2+\mu_3\wedge\bmu_3)\]
defines a K\"ahler metric on $\Sigma$. Using local coordinates $z$ on $\Sigma$ and $\zeta=z_2/z_1$ on $\C\pp^1$, we have
\[\widehat{\omega}= 2(1 + \varphi \overline{\varphi})^2 \, i \ud z \wedge \ud \bar{z}. \]
A direct calculation reveals that the Gauss curvature $\kappa$ of $\widehat{\omega}$ is given by
\[-\kappa\widehat{\omega}=i\pt\bpt\log\rho=\varphi^*\omega_{FS}.\]
Since $\varphi^*\omega_{FS} = \dfrac{\|\nabla\varphi\|^2}{2} \hat{\omega}$, we obtain that
\[\|\nabla\varphi\|^2=-2\kappa,\]
hence $\widehat{\omega}$ has non-positive Gauss curvature.

If we add in the hyperk\"ahler fiber metrics of the fibration $p:X\to\Sigma$, we get a natural Hermitian metric
\[\omega_0=\widehat{\omega}+\alpha\omega_I+\beta\omega_J+\gamma\omega_K\]
on $X$. By definition,
\[ \|\Omega\|_{\omega_0}^2 {\omega_0^3 \over 3!} = i \Omega \wedge \overline{\Omega}. \]
We know that $\|\Omega\|_{\omega_0}$ is a constant by its design since
\[\omega_0^3 = 6 \, \hat{\omega} \wedge {\rm vol}_M, \ \  i \Omega \wedge \bar{\Omega} = 2 \, \hat{\omega} \wedge {\rm vol}_M.  \]
Here ${\rm vol}_M$ is the volume form of $(M,g)$. Furthermore, the balanced condition
\[\ud(\omega_0^2)=0\]
holds as well, and therefore $\omega_0$ solves the conformally balanced equation (\ref{cb}).

Let $f:\Sigma\to\rr$ be any smooth function on $\Sigma$, and define the Hermitian metric
\begin{equation}\label{ansatz}
\omega_f=e^{2f}\widehat{\omega}+e^f(\alpha\omega_I+\beta\omega_J+\gamma\omega_K)
\end{equation}
on $X$. For convenience of notation, we write
\[\omega'=\alpha\omega_I+\beta\omega_J+\gamma\omega_K,\]
whose exterior derivative is given by
\begin{equation}\label{d-omegaprime}
\ud \omega' = \ud \alpha \wedge \omega_I + \ud \beta \wedge \omega_J + \ud \gamma \wedge \omega_K.
\end{equation}
An elementary calculation indicates that
\[\begin{split}& \| \Omega \|_{\omega_f} = e^{-2f} \| \Omega \|_{\omega_0},\\
& \| \Omega \|_{\omega_f} \omega_f^2= 2 (e^f-1)\| \Omega \|_{\omega_0} \hat{\omega} \wedge \omega'+ \| \Omega \|_{\omega_0} \omega_0^2 \end{split}\]
and hence $\omega_f=e^{2f}\widehat{\omega}+e^f\omega'$ solves the conformally balanced equation (\ref{cb}) for arbitrary $f$.

We now consider the anomaly cancellation equation (\ref{ac}). The idea is the following: first we compute the curvature term $\tr(R_f\wedge R_f)$ with respect to the ansatz metric $\omega_f$ and then fix a gauge bundle which solves the Hermitian Yang-Mills equation (\ref{hym}). We then choose $f$ to solve the anomaly cancellation equation. The calculation of the curvature term is essentially the same as we did in \cite{fei2017b}. The result is
\[\tr(R_f\wedge R_f)=i\pt\bpt\left(\frac{\|\nabla\varphi\|^2}{e^f}\omega'\right)+\tr(R'\wedge R'),\]
where $\|\nabla\varphi\|^2=-2\kappa$ is with respect to $\widehat{\omega}$ and $R'$ is the curvature form of the relative cotangent bundle $\Omega_{X/\Sigma}$ with respect to the metric induced from $\omega_0$. Therefore the $\tr(R'\wedge R')$-term can be cancelled by the curvature term from the gauge bundle, and the anomaly cancelation equation (\ref{ac}) reduces to
\begin{equation}
i\pt\bpt\left(\left(e^f+\frac{\alpha'\kappa}{2e^f}\right)\omega'\right)=0.
\end{equation}
This choice of gauge bundle makes sense since the relative cotangent bundle solves the Hermitian Yang-Mills equation (\ref{hym}) automatically for arbitrary $\omega_f$. Roughly speaking this is because the fiber metrics of $p:X\to\Sigma$ are hyperk\"ahler, and we refer to \cite{fei2017b} for more details.

\begin{rmk}~\\
In the case when $M$ is a flat $T^4$, the term $\tr(R'\wedge R')$ disappears \cite{fei2016}. In particular we may use any flat vector bundle as our gauge bundle.
\end{rmk}

Notice that $u=e^f+\dfrac{\alpha'\kappa}{2e^f}$ is a function depending only on $\Sigma$, and we want to solve
\[i\ddb(u\omega')=i\ddb u\wedge\omega'+i\p u\wedge\dbar\omega'-i\dbar u\wedge\p\omega'+u\cdot i\ddb\omega'=0.\]
Direct calculation shows that
\[\begin{split}&\p\omega'=\dbar\alpha\wedge\omega_I+\dbar\beta\wedge\omega_J+\dbar\gamma\wedge\omega_K,\\ &\dbar\omega'=\p\alpha\wedge\omega_I+\p\beta\wedge\omega_J+\p\gamma\wedge\omega_K,\\ &i\ddb\omega'=-i\ddb\alpha\wedge\omega_I-i\ddb\beta\wedge\omega_J-i\ddb\gamma\wedge\omega_K.\end{split}\]
The decomposition of $\ud \omega'$ (\ref{d-omegaprime}) into its $(2,1)$ and $(1,2)$ parts can be seen by acting with the complex structure $J_0$ and using the following identities
\[\begin{split}& J_0 \omega_I = (2 \alpha^2 - 1) \omega_I + 2 \alpha \beta \omega_J + 2 \alpha \gamma \omega_K, \\ & J_0 \omega_J = (2 \beta^2 -1) \omega_J + 2 \beta \alpha \omega_I + 2 \beta \gamma \omega_K,\\ & J_0 \omega_K = (2 \gamma^2-1) \omega_K + 2 \gamma \alpha \omega_I + 2 \gamma \beta \omega_J, \\ & 0 = \alpha \bar{\p} \alpha + \beta \bar{\p} \beta + \gamma  \bar{\p} \gamma.\end{split}\]
Next, a computation shows that $\alpha$, $\beta$ and $\gamma$ all satisfy the equation
\[i\ddb v-\kappa v\hat{\omega}=0,\]
i.e., they live in the kernel of the operator $-\Delta+2\kappa$. Therefore
\begin{equation}\label{positive}
i\pt\bpt\omega'=-\kappa\widehat{\omega}\wedge\omega'
\end{equation}
and the anomaly cancellation equation (\ref{ac}) is further reduced to
\begin{equation}
-\Delta u+2\kappa u=0.
\end{equation}

In conclusion, we have demonstrated that to solve the full Strominger system on generalized Calabi-Gray manifolds, we may use the ansatz (\ref{ansatz}) and the whole system reduces to a quadratic algebraic equation and a linear PDE on $\Sigma$:
\begin{equation}\label{rac}\begin{cases}
&e^f+\dfrac{\alpha'\kappa}{2e^f}=u,\\
&\Delta u-2\kappa u=0.
\end{cases}\end{equation}

In particular, we get a smooth solution if and only if we can find a function $u$ in the kernel of the operator $-\Delta+2\kappa$ such that $u$ is positive at all ramification points of $\varphi$. Consequently the solvability of the Strominger system on $X$ is closely related to spectral properties of the operator $-\Delta+2\kappa=-\Delta-\|\nabla\varphi\|^2$ on $\Sigma$, which we shall explore in the next subsection. It is also clear that there are no solutions to the Strominger system using our ansatz if $\alpha'\leq 0$.

A byproduct of the above calculation is the following statement:
\begin{cor}~\\
A generalized Calabi-Gray manifold $X$ does not admit a pluriclosed metric.
\end{cor}
\begin{proof}
Suppose $\omega_1$ is a pluriclosed metric on $X$, that is, $\omega_1$ is a positive (1,1)-form such that $i\pt\bpt\omega_1=0$. Let us consider the integral
\[\int_X\omega_1\wedge i\pt\bpt\omega'.\]
On one hand, by integration by part, this integral vanishes. On the other hand, $i\pt\bpt\omega'=-\kappa\widehat{\omega}\wedge\omega'$ is positive away from a set of measure zero, hence we get a contradiction.
\end{proof}

This corollary generalizes a result in \cite{fei2016}.

\subsection{Solutions}\label{solutions}

In the previous section we manifested that the solvability of the Strominger system on generalized Calabi-Gray manifolds is closely related to the spectral property of the operator $-\Delta+2\kappa=-\Delta-\|\nabla\varphi\|^2$. In fact, this operator falls into the larger class of Sch\"odinger operators associated with holomorphic maps from Riemann surfaces to complex manifolds.

Let $\varphi:\Sigma\to N$ be a holomorphic map from a compact Riemann surfaces to a complex manifold. By fixing a Hermitian metric on $N$ and choosing a metric $\widehat{\omega}$ on $\Sigma$, we may consider the Sch\"odinger operator
\[L_\varphi=-\Delta-\|\nabla\varphi\|^2\]
on $\Sigma$. Here, $\Delta$ is the Laplace-Beltrami operator associated to $\widehat{\omega}$, and the norm $\|\nabla\varphi\|$ is measured with respect to the chosen metrics on both domain and target. Clearly, $L_\varphi$ is a self-adjoint operator on the Sobolev space $W_1(\Sigma)$ with respect to the $L^2$-norm. The associated quadratic form is
\[Q_\varphi(u,v)=\int_\Sigma\left(\nabla u\cdot\nabla v-\|\nabla\varphi\|^2uv\right)\widehat{\omega}.\]
Since the domain has real dimension two, this quadratic form is conformally invariant. Therefore, the kernel of $L_\varphi$ and the number $\ind~L_\varphi$ of negative eigenvalues of $L_\varphi$, known as the index, do not depend on the choice of the metric $\widehat{\omega}$.

For example, if $\varphi$ is the (extended) Gauss map of a minimal surface $\Sigma$ in flat 3-space, then $L_\varphi=-\Delta+2\kappa$ is the Jacobi operator, or the stability operator of the minimal surface, which comes from the second variation of the area functional. As mentioned in \cite{montiel1991}, this operator also shows up naturally in the study of Willmore surfaces and Polyakov quantum string theory. Throughout this paper, the target $N$ is always the projective line $\ccc\pp^1$ equipped with the Fubini-Study metric (\ref{fs}).

Now let $\varphi:\Sigma\to\ccc\pp^1$ be a nonconstant holomorphic map. We would like to understand the space $\ker L_\varphi$. By the stereographic projection (\ref{stereographic}), we may think of $\varphi=(\alpha,\beta,\gamma)$ as a map into $\rr^3$. It is easy to check that $\alpha$, $\beta$ and $\gamma$ live in the kernel of $L_\varphi$, therefore
\[\dim\ker L_\varphi\geq 3.\]
Let $V_\varphi$ denote the 3-dimensional space spanned by $\alpha$, $\beta$ and $\gamma$. Montiel and Ros \cite{montiel1991} showed that the space $\ker L_\varphi/V_\varphi$ is in 1-1 correspondence with complete branched minimal immersions into $\rr^3$ with finite total curvature and planar ends (up to translation) such that its extended Gauss map is $\varphi$. However, it seems that in general we do not know how to compute $\ker L_\varphi$ and there are examples where $\dim L_\varphi$ is strictly greater than 3. For this reason, to solve the reduced the Strominger system (\ref{rac}) on generalized Calabi-Gray manifolds, let us consider only the case $u\in V_\varphi$. In this scenario, the condition that $u$ is positive at all ramification points of $\varphi$ is equivalent to all branched points of $\varphi$ on $\ccc\pp^1$ lying in an open hemisphere, which we abbreviate as the ``hemisphere condition''. Clearly if $(\Sigma,\varphi)$ satisfies the hemisphere condition and $\tau:\widetilde{\Sigma}\to\Sigma$ is an unramified covering of Riemann surfaces, then $(\widetilde{\Sigma},\varphi\circ\tau)$ also satisfies the hemisphere condition.

Given any vanishing spinorial pair $(\Sigma,\varphi)$, we may make the hemisphere condition holds by composing $\varphi$ with a suitable automorphism of $\ccc\pp^1$, since there exist M\"obius transformations on $\ccc\pp^1$ pushing all points away south pole to the north hemisphere. As a consequence, there exist vanishing spinorial pairs satisfying the hemisphere condition for every genus $g\geq3$ and the moduli of curves with vanishing spinorial pairs satisfying the hemisphere condition is exactly the moduli of curves with basepoint-free theta characteristics, which roughly speaking forms a Zariski open set of the theta-null divisor in the moduli space of curves \cite{teixidor1987}. However, it seems unknown whether there exists a minimal surface in $T^3$ with its Gauss map satisfying the hemisphere condition. The hemisphere condition fails for the 5-dimensional family of triply periodic minimal surfaces constructed by Meeks \cite{meeks1990}.

Summarizing our previous results, we have proved that

\begin{thm}\label{main}~\\
Let $(\Sigma,\varphi)$ be a vanishing spinorial pair with the hemisphere condition satisfied. Then we may construct explicit solutions to the Strominger system on the associated generalized Calabi-Gray manifold $X$. In fact, we get a family of such solutions of real dimension $\dim\ker L_\varphi\geq3$. As a consequence, for every genus $g\geq3$, there exist smooth solutions to the Strominger system on genus $g$ generalized Calabi-Gray manifolds. They have infinitely many distinct topological types and sets of Hodge numbers.
\end{thm}

The genus $g=3$ case is of particular interest. Assuming that $\Sigma$ is of genus 3 and $(\Sigma,\varphi)$ is a vanishing spinorial pair, then clearly $\varphi:\Sigma\to\ccc\pp^1$ is of degree 2, therefore $\Sigma$ is hyperelliptic. In this case, the map $\varphi:\Sigma\to\ccc\pp^1$ is a double covering branched over 8 points on $\ccc\pp^1$. Conversely, from any 8 distinct points on $\ccc\pp^1$, one can construct a double covering $\varphi:\Sigma\to\ccc\pp^1$ branched over them. By Riemann-Hurwitz, such $\Sigma$'s are hyperelliptic genus $3$ curves. In addition, from the double covering construction, the pullback condition $\varphi^*\clo(2)=K_\Sigma$ is automatically satisfied. Therefore we can construct vanishing spinorial pairs satisfying the hemisphere condition on any hyperelliptic genus 3 curve, hence solving the Strominger system.

It is also worth pointing out that Theorem \ref{main} can be strengthened to an ``if and only if'' statement for the genus 3 case. This is because when the genus is 3 and $\varphi:\Sigma\to\ccc\pp^1$ is the hyperelliptic covering, we may use the hyperelliptic involution $\iota$ on $\Sigma$ commuting with $\varphi$. Since $\iota$ is an isometry of $\widehat{\omega}$, it also acts on the space $\ker L_\varphi$. Consequently $\ker L_\varphi$ breaks up into eigenspaces of $\iota$ with eigenvalue 1 and -1:
\[\ker L_\varphi=V_1\oplus V_{-1}.\]
The eigenspace $V_1$ of eigenvalue 1 consists of functions pulled back from $\ccc\pp^1$, hence we have
\[V_1=V_\varphi,\]
since functions in $V_1$ have to be the first eigenfunctions of the spherical Laplacian on $\ccc\pp^1$. On the other hand $\iota$ acts as -1 on $V_{-1}$, therefore any function in $V_{-1}$ must vanish on the fixed points of $\iota$, i.e., the ramification points of $\varphi$, hence we cannot use functions in $V_{-1}$ to produce smooth solutions of the Strominger system. So we have proved
\begin{cor}~\\
Suppose $\Sigma$ has genus 3 and let $(\Sigma,\varphi)$ be a vanishing spinorial pair. Using our ansatz, the Strominger system has a smooth solution if and only if $(\Sigma,\varphi)$ satisfies the hemisphere condition.
\end{cor}

\subsection{Geometric Consequences of the Hemisphere Condition}

Though the hemisphere condition is trivial for algebraic geometry, it has nontrivial consequences in differential geometry. In this subsection, we prove a simple result regarding the index of the operator $-\Delta-\|\nabla\varphi\|^2$ mentioned in Subsection \ref{solutions}. We refer to \cite{grigoryan2004, fei2018} for related estimates of eigenvalues.

Let $\varphi:\Sigma\to\ccc\pp^1$ be a holomorphic map from a compact Riemann surface to the projective line. Fix the Fubini-Study metric on $\ccc\pp^1$. As we have seen, the index $\ind~L_\varphi$, i.e., the number of negative eigenvalues of the operator $-\Delta-\|\nabla\varphi\|^2$, does not depend on the choice of the metric on $\Sigma$. Moreover, we have the following estimates
\begin{thm}(Tysk \cite{tysk1987})
\[\ind~L_\varphi\leq 7.68183\cdot\deg\varphi.\]
\end{thm}
\begin{thm}(Grigor'yan-Netrusov-Yau \cite{grigoryan2004})
\[\ind~L_\varphi\geq C\cdot\deg\varphi\]
for some absolute constant $C$.
\end{thm}
Assuming the hemisphere condition, we can derive
\begin{prop}\label{index}~\\
If the branched points of $\varphi$ in $\ccc\pp^1$ all lie in an open hemisphere, then we have the estimate
\[\ind~L_\varphi\geq\deg\varphi.\]
\end{prop}
\begin{proof}
As mentioned in Subsection \ref{solutions}, the index $\ind~L_\varphi$ does not depend on the choice of metrics on $\Sigma$. In particular, we may use the singular metric $\omega_\varphi:=\varphi^*\omega_{FS}$ to calculate $\ind~L_\varphi$, because the eigenvalues of the corresponding operator
\[-\Delta_{\varphi} +\|\nabla\varphi\|^2_{\omega_\varphi} = -\Delta_\varphi+2\]
can still be defined by variational characterization \cite{tysk1987}, where $\Delta_\varphi$ is the singular Laplace-Beltrami operator associated to $\omega_\varphi$. Consequently we only need to estimate the number of eigenvalues of $-\Delta_\varphi$ less than 2.

Without loss of generality, we may assume all branched points of $\varphi$ lie in the south hemisphere $S$. Let $N$ be the north hemisphere, then we know that $\varphi^{-1}(N)$ consists of $\deg\varphi$ copies of disks and $\varphi^{-1}(S)$ is a connected compact Riemann surface with $\deg\varphi$ copies of $S^1$ as its boundary. Applying \cite[Lemma 12]{montiel1991} to this decomposition of $\Sigma$, and using the counting result from \cite[Lemma 11(b)]{montiel1991}, we prove the proposition.
\end{proof}
\begin{rmk}~\\
The hemisphere condition is necessary for the above proposition, otherwise we have the counterexample due to Ross \cite{ross1992}, where he showed that the Schwarz' P and D surfaces (genus 3 minimal surfaces in flat $T^3$) have $\ind~L_\varphi=1$. Proposition \ref{index} can also be proved by using \cite[Lemma 20]{montiel1991}.
\end{rmk}

\subsection{Some Uniqueness Results}

In general, it is hard to establish uniqueness results about the Strominger system. One typically needs to assume that the solutions take the form of a certain ansatz: see for example the Fu-Yau solutions. In our case, one can easily check that any warped product metric on $X$ which solves the conformally balanced equation (\ref{cb}) must live in the family (\ref{ansatz}). Therefore in this subsection, we only consider solutions to the Strominger system on generalized Calabi-Gray manifolds which are of the form (\ref{ansatz}).

As we have seen in previous subsections, our solution depends on a function $u\in\ker L_\varphi$ such that $u$ is positive at all ramification points of $\varphi:\Sigma\to\ccc\pp^1$. Suppose $\dim_\rr\ker L_\varphi=n\geq 3$ and choose a basis $\{u_1,\dots,u_n\}$ of $\ker L_\varphi$. Let
\[D=\{u\in\spann\{u_1, u_2, \dots, u_n\}:u>0\textrm{ at branched points of }\varphi\}.\]
The set $D$ forms an open convex polyhedral cone in $\ker L_\varphi$. Consider the map $T:D\to\rr^n$ given by
\[T(u)=\left(\int_\Sigma e^fu_1\hat{\omega},\dots,\int_\Sigma e^fu_n\hat{\omega}\right),\]
where $e^f$ is determined from $u$ by the relation
\begin{equation} \label{e^f-to-u}
e^f=\frac{1}{2}\left(u+\sqrt{u^2-2\alpha'\kappa}\right)>0.
\end{equation}
Any $u \in D$ can be expressed as $u=t^i u_i$, and we will use the coordinates $t^i$ to compute the Jacobian of $T$.
\begin{prop}\label{convex}~\\
There exists a strictly convex function $F:D\to\rr$ such that $T(u)=\nabla F(u)$. As a consequence, $T$ is a diffeomorphism from $D$ to its image.
\end{prop}
\begin{proof}
We calculate the Jacobian matrix of $T$, which is given by
\[J(T)(u)=\frac{1}{2}\int_\Sigma\left(1+\frac{u}{\sqrt{u^2-2\alpha'\kappa}}\right)\begin{pmatrix}u_1^2 &\dots &u_1u_n\\ \vdots& \ddots& \vdots\\ u_nu_1& \dots& u_n^2\end{pmatrix}\hat{\omega} .\]
Since the $u_i$ are linearly independent, it follows that $J(T)$ is symmetric positive-definite. This implies, at least locally, the existence of the potential function $F$. In fact, we can write down $F$ explicitly as
\[F(u)=\frac{1}{2}\int_\Sigma \left(e^{2f}-\alpha'\kappa f\right)\hat{\omega},\]
where we use (\ref{e^f-to-u}) to relate $e^f$ to $u$. The second part of the proposition follows directly from properties of strictly convex functions.
\end{proof}

As our ansatz solves the conformally balanced equation (\ref{cb}), the closed 4-form $\|\Omega\|_\omega\cdot\omega^2$ defines a de Rham cohomology class. In fact, this class is given by
\[\begin{split}\left[\|\Omega\|_{\omega_f}\cdot\omega_f^2\right]&=2\sqrt{2}\left([\vol_M]+[\omega_I][e^f\alpha\hat{\omega}]
+[\omega_J][e^f\beta\hat{\omega}]+[\omega_K][e^f\gamma\hat{\omega}]\right)\\ &\in H^4(X;\rr)\end{split}\]
under the K\"unneth isomorphism
\[H^*(X;\rr)\cong H^*(M;\rr)\otimes H^*(\Sigma;\rr),\]
which is determined by the integrals
\[\int_\Sigma e^f\alpha\widehat{\omega},\quad \int_\Sigma e^f\beta\widehat{\omega}, \quad\textrm{ and }\quad \int_\Sigma e^f\gamma\widehat{\omega}.\]
One motivation for looking for solutions inside a prescribed class
\[\left[\|\Omega\|_{\omega}\cdot\omega^2\right] \in H^4(X;\mathbb{R})\]
is that this class is preserved by the Anomaly flow \cite{phong2018b}.

We know that there is a 3-dimensional subspace $V_\varphi\subset\ker L_\varphi$, spanned by $\{\alpha,\beta,\gamma\}$, corresponding to rotational Jacobian fields. It is expected, as partially supported by results in \cite{nayatani1993}, that for a generic vanishing spinorial pair $(\Sigma,\varphi)$, the equality $\dim\ker L_\varphi=\dim V_\varphi=3$ holds. Assuming so, combining these observations with Proposition \ref{convex} leads to the following uniqueness result:
\begin{cor}~\\
Assuming the ansatz (\ref{ansatz}) and $\dim\ker L_\varphi=3$, then there is at most one solution to the Strominger system in each de Rham cohomology class in $H^4(X;\rr)$.
\end{cor}

\bigskip
\noindent Department of Mathematics, Columbia University, New York, NY 10027, USA

\bigskip
\noindent tfei@math.columbia.edu\\
zjhuang@math.columbia.edu\\
picard@math.columbia.edu

\end{document}